\documentclass[a4,12pt]{amsart}
\usepackage{amssymb,amstext,amsmath,amscd,amsthm,amsfonts,enumerate,latexsym, comment}
\usepackage{color}
\usepackage[dvipdfmx]{graphicx}
\usepackage[all]{xy}

\usepackage{url}
\theoremstyle{plain}
\newtheorem{thm}{Theorem}[section]

\newtheorem*{thm*}{Theorem}
\newtheorem*{cor*}{Corollary}

\newtheorem{prop}[thm]{Proposition}

\newtheorem{lem}[thm]{Lemma}
\newtheorem{cor}[thm]{Corollary}

\newtheorem{claim}{Claim}
\newtheorem*{claim*}{Claim}

\theoremstyle{definition}
\newtheorem{defn}[thm]{Definition}

\newtheorem{ex}[thm]{Example}

\newtheorem{rem}[thm]{Remark}

\newtheorem{fact}[thm]{Fact}

\newtheorem{setup}[thm]{Setup}

\theoremstyle{remark}

\newtheorem*{proof of claim}{{\sl Proof of Claim}}

\numberwithin{equation}{thm}

\def\Im{\operatorname{Im}}

\def\Hom{\operatorname{Hom}}

\def\Max{\operatorname{Max}}

\def\Im{\mathrm{Im}}

\def\tr{\mathrm{tr}}

\def\rank{\mathrm{rank}}

\def\m{\mathfrak m}

\newcommand{\rme}{\mathrm{e}}

\newcommand{\rmQ}{\mathrm{Q}}

\newcommand{\fkm}{\mathfrak{m}}

\newcommand{\fkp}{\mathfrak{p}}

\newcommand{\mapright}[1]{%
\smash{\mathop{%
\hbox to 1cm{\rightarrowfill}}\limits^{#1}}}

\newcommand{\mapleft}[1]{%
\smash{\mathop{%
\hbox to 1cm{\leftarrowfill}}\limits_{#1}}}

\def\Ass{\operatorname{Ass}}

\def\ol{\overline}

\def\tr{\mathrm{tr}}

\title{Reflexive modules over Arf local rings}

\author{Ryotaro Isobe}
\address{General Education Division, National Institute of Technology, Oshima College, 1091-1, Oazakomatsu, Suooshima-cho, Oshima-gun, Yamaguchi, 742-2193, Japan}
\email{isobe.ryotaro@oshima-k.ac.jp}

\author{Shinya Kumashiro}
\address{National Institute of Technology, Oyama College, 771 Nakakuki, Oyama, Tochigi, 323-0806, Japan}
\email{skumashiro@oyama-ct.ac.jp}
\email{shinyakumashiro@gmail.com}

\thanks{The first author was supported by JSPS KAKENHI Grant Number JP21K13767. The second author was supported by JSPS KAKENHI Grant Number JP21K13766 and by Grant for Basic Science Research Projects from the Sumitomo Foundation (Grant number 2200259).}

\begin{document}

\begin{abstract}
We provide a certain direct-sum decomposition of reflexive modules over (one-dimensional) Arf local rings. We also see the equivalence of three notions, say, integrally closed ideals, trace ideals, and reflexive modules of rank one (i.e., divisorial ideals) up to isomorphisms in Arf rings. As an application, we obtain the finiteness of indecomposable reflexive modules, up to isomorphism, for analytically irreducible Arf local rings.
\end{abstract}

\maketitle

\section{Introduction}\label{section1}

The study of maximal Cohen-Macaulay modules over Cohen-Macaulay local rings, which is known as the Cohen-Macaulay representation theory, is an important and a classical subject. We cannot cover all literature comprehensively in this article; instead we refer the reader to the book such as \cite{LW}. One problem in this study is when a Cohen-Macaulay local ring has a finite number of irreducible maximal Cohen-Macaulay modules up to isomorphism (such a ring is called {\it of finite Cohen-Macaulay type}). In dimension one, there is a complete classification of rings of finite Cohen-Macaulay type (\cite[4.10 Theorem]{LW}). In relation to the result, Bass proved in his ``ubiquity'' paper \cite{Ba} that the following hold for rings with multiplicity two. 

\begin{fact} {\rm (\cite[4.18 Theorem]{LW})}\label{f11}
Let $R$ be a Cohen-Macaulay local ring of dimension one, with multiplicity two. Then the following hold true. 
\begin{enumerate}[\rm(a)] 
\item Every maximal Cohen-Macaulay module is isomorphic to a direct sum of ideals of $R$. 
\item The ring $R$ has finite Cohen-Macaulay type if and only if $R$ is analytically unramified, that is, the completion of $R$ is reduced. 
\end{enumerate}
\end{fact}




In this article, by focusing on reflexive modules, we generalize the above result from rings with multiplicity two to Arf local rings. For the definition of reflexive modules, refer Setup \ref{setup1}. Here, we just note that in dimension one, all reflexive modules are maximal Cohen-Macaulay modules and both coincides for rings with multiplicity two. 
We now explain on Arf rings. To simplify the description, let $(R, \fkm)$ be a one-dimensional Cohen-Macaulay local ring. Then, $R$ is an {\it Arf ring} if and only if $R$ satisfies the following condition (\cite{L}): 
for every $\fkm$-primary integrally closed ideal $I$, there exists $a\in I$ such that $I^2=aI$.

\vskip.5\baselineskip

The notion of Arf rings originates from the classification of certain singular points of plane curves by Arf \cite{A}. Thereafter, in 1971 Lipman generalized the notion for one-dimensional semi-local rings and characterized as above (\cite{L}). The study of Arf rings still has been explored, see, for example, \cite{CCCEGIM, Is}.
The aim of this study is to clarify the structure of reflexive modules over Arf local rings, and the main result is as follows. For the definition of trace ideals, see Subsection \ref{subsec22}.

\begin{thm} \label{mainthm} {\rm (Theorems \ref{equiv} and \ref{main2})}
Let $R$ be an Arf local ring with the maximal ideal $\fkm$. Then, the following hold true.
\begin{enumerate}[{\rm (a)}] 
\item For an $\fkm$-primary ideal $I$, the following are equivalent:
\begin{enumerate}[{\rm (i)}]  
\item $I$ is reflexive as an $R$-module.
\item $I$ is isomorphic to some trace ideal.
\item $I$ is isomorphic to some integrally closed ideal.
\end{enumerate} 
\item Suppose that the integral closure $\ol{R}$ of $R$ is local. Then, every reflexive module with positive rank is isomorphic to a direct sum of $\fkm$-primary integrally closed ideals of $R$.
\end{enumerate} 
\end{thm}

The class of Arf rings contains all Cohen-Macaulay local rings with multiplicity two (\cite[Example, p.664]{L}). 
Furthermore, Cohen-Macaulay local rings with multiplicity two are Gorenstein; hence, all maximal Cohen-Macaulay modules are reflexive. 
Therefore, Theorem \ref{mainthm}(b) fully extends Fact \ref{f11} from rings with multiplicity two to Arf rings, under the additional assumption that $\ol{R}$ is local. As an application to the above result, we further obtain the finiteness of indecomposable reflexive modules over Arf local domains (Corollary \ref{corcor}). We should note that the finiteness result has been independently announced in \cite{D, DL} at about the same time (but the proof seems different).


\vskip.5\baselineskip

We now explain how this article is organized. In Section \ref{section2}, we collect facts on trace ideals and Arf rings to prove our main result. In Section \ref{section3}, we establish Theorem \ref{mainthm}. Below, we summarize the notations that we used in this article.

\begin{setup}\label{setup1}
In what follows, all rings are commutative Noetherian rings and all modules are finitely generated. 
Let $R$ be a ring and $M$ an $R$-module. $\rmQ(R)$ and $\ol{R}$ denote the total ring of fraction of $R$ and the integral closure of $R$, respectively. $M$ is called {\it torsionfree} if $M\to \rmQ(R)\otimes_R M$ is injective. We denote by $(-)^*$ the $R$-dual $\Hom_R(-, R)$. We consider the canonical map 
\begin{center}
$h:M\to M^{**}; x\mapsto (f\mapsto f(x))$, where $x\in M$ and $f\in M^*$. 
\end{center}
$M$ is called {\it torsionless} (resp. {\it reflexive}) if $h$ is injective (resp. bijective) (\cite[page 19]{BH}). In our assumption on $R$ and $M$, the torsionless of $M$ implies that the torsionfreeness of $M$ (see \cite[Excercise 1.4.20]{BH}). 
For each non-negative integer $\ell$, we say that $M$ has {\it rank $\ell$} if $M_\fkp$ is an $R_\fkp$-free module of rank $\ell$ for all $\fkp\in \Ass R$. 


For an ideal $I$, $\ol{I}$ denotes the {\it integral closure} of $I$.
We say that $I$ is a {\it fractional ideal} if $I$ is a finitely generated $R$-submodule of $\rmQ(R)$ such that $I$ contains a non-zerodivisor of $R$. 
For fractional ideals $I$ and $J$, $I:J$ denotes the colon fractional ideal $\{\alpha\in \rmQ(R) \mid \alpha J\subseteq I\}$. We freely use the fact that $I:J\cong \Hom_R(J, I)$ (see \cite[Lemma 2.4.2]{SH}). 
\end{setup}

\section{Preliminaries}\label{section2}

In this section, we summarize several known results that we need in this article.

\subsection{trace ideals}\label{subsec22}

Let $R$ be a Noetherian ring, and let $M$ be a finitely generated $R$-module. 
Then, 
\begin{align*} 
\mathrm{tr}_R(M)=& \sum_{f\in M^*} \Im f \\
=& \Im (M^* \otimes_R M \to R)\text{, where $f\otimes x\mapsto f(x)$ for $f\in M^{*}$ and $x\in M$, }
\end{align*}
is called the {\it trace ideal} of $M$. An ideal $I$ is called a {\it trace ideal} if $I=\mathrm{tr}_R(M)$ for some $R$-module $M$. 
For an ideal $J$ containing a non-zerodivisor of $R$, we have 
\[
\tr_R(J)=(R:J)J
\] 
by identifying the map $J^* \otimes_R J \to R$ with $(R:J) \otimes_R J \to R; f\otimes x\mapsto fx$ for $f\in R:J$ and $x\in J$. 


\begin{lem} {\rm (\cite{GIK2, LP})}\label{tuesday}
Let $I$ be an ideal of $R$ and $\iota:I\to R$ be the embedding. Then, the following are equivalent:
\begin{enumerate}[{\rm (a)}] 
\item $I$ is a trace ideal, 
\item $I=\mathrm{tr}_R(I)$, and 
\item $\Hom_R(I, I) \xrightarrow{\iota_*} \Hom_R(I, R)$ is bijective.
\end{enumerate} 	
\end{lem}

\begin{proof} 
This is known, but we include a proof for the convenience of reader. 

(c) $\Rightarrow$ (b) $\Rightarrow$ (a) is clear.  

(a) $\Rightarrow$ (c): Suppose that there exists $f\in \Hom_R(I, R)$ such that $\Im f\not\subseteq I$, and choose $x\in \Im f\setminus I$. Set $I=\mathrm{tr}_R(M)$ for some $R$-module $M$. Then, the map $M^*\otimes_R M \to I \xrightarrow{f} R$ shows that $x\in \mathrm{tr}_R (M)=I$. This is a contradiction.
\end{proof}

\begin{cor}\label{ab2.2}
Let $I$ be an ideal containing a non-zerodivisor of $R$. 
Then, the following are equivalent:
\begin{enumerate}[{\rm (a)}] 
\item $I$ is a trace ideal, 
\item $I=(R:I)I$, and 
\item $I:I=R:I$.
\end{enumerate} 
\end{cor}

\begin{rem} \label{rem2.3} 
\begin{enumerate}[\rm(a)]
\item {\rm (\cite[Proposition 2.8(viii)]{Lin})}: $\tr_R(M)\otimes_R A =\tr_A(M\otimes_R A)$ for any commutative flat $R$-algebra $A$. In particular, the localization of a trace ideal is a trace ideal. 
\item If $M$ has a positive rank, then $\mathrm{tr}_R(M)$ contains a non-zerodivisor of  $R$. 
\end{enumerate}
\end{rem}

\begin{proof} 
(a): Since $\mathrm{tr}_R(M)$ is the image of  $M^* \otimes_R M \to R$, where $f\otimes x\mapsto f(x)$, flat extensions are compatible. 

(b): Let $\fkp\in \Ass R$. By (1), $\tr_R(M)R_\fkp=\tr_{R_\fkp}(M_\fkp)$. Furthermore, if $M$ has a positive rank $\ell$, then $\tr_{R_\fkp}(M_\fkp)=\tr_{R_\fkp}(R_\fkp^{\ell})=R_\fkp$. It follows that $\tr_R(M)R_\fkp=R_\fkp$, which implies that $\tr_R(M)\not\subseteq \fkp$. Hence, $\mathrm{tr}_R(M)$ contains a non-zerodivisor of  $R$. 
\end{proof}

\begin{prop} {\rm (cf. \cite[(7.2) Proposition]{Ba})}\label{reflexive}
Let $M$ be a reflexive $R$-module of positive rank $\ell$. We define an endomorphism algebra $A:=\mathrm{tr}_R(M): \mathrm{tr}_R(M)$. Then, we can regard $M$ as an $A$-module by the action extending the $R$-action.
\end{prop}

\begin{proof} 
By applying the $R$-dual $(-)^*$ to the canonical map $M^*\otimes_R M \to \tr_R(M) \to 0$ (see the definition of trace ideals), we obtain that 
\[
0 \to \Hom_R(\tr_R (M), R) \to \Hom_R(M, M^{**})=\Hom_R(M, M)
\] 
since $M$ is reflexive. On the other hand, we have 
\[
\Hom_R(\tr_R (M), R)\cong \Hom_R(\tr_R (M), \tr_R(M))\cong A
\] 
by Lemma \ref{tuesday}. Hence, passing to the map $A\to \Hom_R(M, M)$, we can regard $M$ as an $A$-module. 
We note that if we restrict the $A$-action to $R$, then it coincides with the $R$-action. 
\end{proof}

\begin{rem} 
With the same assumption and notation of Proposition \ref{reflexive}, the $A$-action extending the $R$-action of $M$ is unique.
\end{rem}

\begin{proof} 
Assume that there are two actions $\circ$ and $\circ'$ extending the $R$-action of $M$. Let $\alpha\in A$ and $x\in M$, and set $\alpha=a/s$, where $a,s\in R$ and $s$ is a non-zerodivisor of $R$. Then,
\begin{align*} 
s(\alpha\circ x)=(s\alpha)\circ x = a\circ x = 
a\circ' x = (s\alpha)\circ' x =s(\alpha\circ' x).
\end{align*}
It follows that $s(\alpha\circ x - \alpha\circ' x)=0$. This proves that $\alpha\circ x = \alpha\circ' x$. Indeed, all non-zerodivisors of $R$ are non-zerodivisors of $M$ since $\Ass M=\Ass M^{**} \subseteq \Ass R$.
\end{proof}

Proposition \ref{reflexive} fails if $M$ is a non-reflexive torsionfree module. 

\begin{ex} 
Let $R=K[[t^3, t^4, t^5]]$ and $I=(t^3, t^4)$, where $K[[t]]$ denotes the formal power series ring over a filed $K$. Then, since $I$ is isomorphic to the canonical module of $R$ and $R$ is not Gorenstein (\cite[Example 2.1.9]{GW}), $I$ is a non-reflexive torsionfree $R$-module. 
It is easy to check that $\tr_R(I)$ contains the maximal ideal $\fkm=(t^3, t^4, t^5)$, thus $\tr_R(I): \tr_R(I)=\fkm:\fkm=K[[t]]$. However, $I\ne I K[[t]]$, that is, $I$ does not become a $K[[t]]$-module.
\end{ex}

Later we will use the following fact. 

\begin{prop} {\rm(\cite[Proposition 2.8(iii)]{Lin})}\label{lindo}
Suppose that $R$ is a local ring. Then, 
$\tr_R(M)=R$ if and only if $M$ has a nonzero free direct summand.
\end{prop}

\subsection{Arf rings}

In this subsection we survey the notion of Arf rings. Let $R$ be a Cohen-Macaulay semi-local ring with $\dim R_\fkm=1$ for all $\fkm\in \Max R$. Here, we should not restrict ourselves to local rings. Indeed, the Arf property is often inherited by intermediate rings between $R$ and $\ol{R}$ as shown in Theorem \ref{2.7}, but intermediate rings are not necessarily local rings even if $R$ is local. 


\begin{defn}[\cite{L}]\label{2.2}
$R$ is called an {\it Arf ring} if the following two conditions hold:
\begin{enumerate}[{\rm (a)}] 
\item
For every integrally closed ideal $I$ which contains a non-zerodivisor of $R$, there exists an element $a\in I$ such that $I=\overline{(a)}$, i.e., $I^{n+1}=aI^n$ for some $n\ge0$.  
\item
If $x, y, z \in R$ such that $x$ is a non-zerodivisor of $R$ and $\frac{y}{x}, \frac{z}{x}\in \overline{R}$, then $\frac{yz}{x}\in R$.
\end{enumerate}
\end{defn}

\begin{rem}{\rm (\cite[Corollary 2.5]{L})} \label{maximal}
$R$ is Arf if and only if $R_\fkm$ is Arf for all maximal ideals $\fkm$ of $R$.
\end{rem}

The notion of Arf rings is characterized as follows.

\begin{thm} [{\cite[Theorem 2.2]{L}}]\label{2.3}
 Let $R$ be a Cohen-Macaulay semi-local ring with $\dim R_\fkm=1$ for all $\fkm\in \Max R$. Then, the following conditions are equivalent:
\begin{enumerate}[{\rm (a)}] 
\item $R$ is an Arf ring.
\item For every integrally closed ideal $I$ that contains a non-zerodivisor of $R$, there exists $a\in I$ such that $I^2=aI$.
\end{enumerate}
\end{thm}

According to the condition that $I^2=aI$, we note the following fact. 

\begin{fact}(\cite[Lemma 1.11]{L})\label{am1142}
Let $a\in I$ be a reduction of $I$, that is, $I^{n+1}=aI^n$ for some $n>0$. Then the following are equivalent. 
\begin{enumerate}[\rm(1)] 
\item $I^2=aI$.
\item $I:I = a^{-1}I$. 
\end{enumerate}
Moreover, if the above condition holds, then $I^2=bI$ for any reduction $b\in I$. 
\end{fact}

Let $\operatorname{J}(R)$ denote the Jacobson radical of $R$. We note that $\operatorname{J}(R)$ is an integrally closed ideal containing a non-zerodivisor of $R$ (\cite[Remark 1.1.3(6)]{SH}). Set 
\[
\displaystyle R_1=\bigcup_{i\ge0}[\operatorname{J}(R)^i:\operatorname{J}(R)^i], \quad 
\text{and define recursively} \quad 
R_n=
\begin{cases}
\  R & \ \ (n=0)\\
\ [R_{n-1}]_1 & \ \ (n>0)\\
\end{cases}
\]
for each $n\ge 0$. Then, we obtain the tower
$$R=R_0\subseteq R_1\subseteq \cdots \subseteq R_n \subseteq \cdots$$
of rings in $\ol{R}$. Every $R_n$ is a Cohen-Macaulay semi-local ring such that all maximal ideals are of height one. 
By using this tower of rings, we obtain another characterization of Arf rings as follows.

\begin{thm}[{\cite[Theorem 2.2]{L}}]\label{2.7}
The following conditions are equivalent:
\begin{enumerate}[{\rm (a)}]
\item
$R$ is an Arf ring.
\item
$[R_n]_\m$ has maximal embedding dimension for every $n\ge0$ and $\m\in\Max R_n$, that is, the embedding dimension coincides with the multiplicity.
\end{enumerate}
\end{thm} 

The following examples can be verified by using Theorem \ref{2.7}.

\begin{ex}[{\cite[Example, p.664]{L}}]\label{2.8}
Let $R$ be a Cohen-Macaulay local ring with $\rme(R)\le 2$. 
Then, $R$ is an Arf ring.
\end{ex}

\begin{ex} \label{2.13}
Let $n>0$, and let $R=K[[t^n, t^{n+1}, \dots, t^{2n-1}]]$ be a subring of the formal power series ring $K[[t]]$. Then, $R$ is an Arf ring. This shows that for any $n>0$, there exist Arf rings with multiplicity $n$.
\end{ex}

Note that Example \ref{2.13} shows that the main result of this article, Theorem \ref{mainthm}, fully generalizes Fact \ref{f11}. We can find more numerous examples of Arf rings, for example, by using a characterization of numerical semigroup Arf rings via numerical semigroup (\cite[Corollary 3.19]{RG}).




\section{Main results}\label{section3}
In this section we prove theorems displayed in the introduction. Throughout this section, we work under the following assumption. 
\begin{setup}
$R$ denotes an Arf ring. 
\end{setup}

We note that $R$ is a one-dimensional semi-local ring by the definition of Arf rings. 
The following is a key of our result.

\begin{prop} \label{traceint}
If $I$ is a trace ideal containing a non-zerodivisor of $R$, then $I$ is an integrally closed ideal.
\end{prop}

\begin{proof} 
By passing to the localization of each maximal ideal containing $I$, we may assume that $R$ is a local ring with maximal ideal $\fkm$ (see Remarks \ref{rem2.3}(a) and \ref{maximal}, \cite[Proposition 1.1.4]{SH}). 
Let $T=R[X]_{\fkm R[X]}$, where $R[X]$ is the polynomial ring over $R$. 
Then, $IT$ is a trace ideal in $T$ by Remark \ref{rem2.3}(a). Since the residue field $T/\fkm T$ is infinite, there exists a non-zerodivisor $t\in IT$ such that $tT$ is a reduction of $IT$. It is sufficient to show that $IT=\ol{I}T$ since $R\to T$ is faithfully flat. 

Because $IT\subseteq \ol{I}T$, we obtain that 
\[
T\subseteq \ol{I}T: \ol{I}T\subseteq T:\ol{I}T \subseteq T:IT = IT:IT, 
\]
where the last equality follows from Corollary \ref{ab2.2}. 
By multiplying $IT$ to the above, 
\[
IT\subseteq (\ol{I}T: \ol{I}T)IT \subseteq (IT:IT)IT=IT.
\] 
It follows that $IT = (\ol{I}T: \ol{I}T)IT$. 
Meanwhile, there exists a non-zerodivisor $a\in \ol{I}$ of $R$ such that 
$\ol{I}^2=a\ol{I}$ because $R$ is Arf. Thus, $(\ol{I}T)^2=a(\ol{I}T)$. By Fact \ref{am1142}, we have $(\ol{I}T)^2=t(\ol{I}T)$.  
Hence, it follows that 
\[
\ol{I}T:\ol{I}T = t^{-1} \ol{I}T
\]
by Fact \ref{am1142}. Therefore, we obtain that 
\[
IT=(\ol{I}T: \ol{I}T)IT=(t^{-1} \ol{I}T)IT. 
\]
On the other hand, since $t\in IT$, we have $t\ol{I}T\subseteq IT\cdot\ol{I}T \subseteq (\ol{I}T)^2=t\ol{I}T$. Hence, $(t^{-1} \ol{I}T)IT =t^{-1} t\ol{I}T =\ol{I}T$. Thus, we obtain that $IT = \ol{I}T$ as desired. 
\end{proof}

The converse of Proposition \ref{traceint} does not hold true. 

\begin{ex} 
Let $R=K[[t^2, t^3]]$ and $I=(t^3, t^4)$, where $K[[t]]$ denotes the formal power series ring over a field $K$. Then, $R$ is of multiplicity two; hence, $R$ is an Arf ring. Furthermore, $I=t^3K[[t]]=t^3\ol{R}$ is an integrally closed ideal of $R$. However, $\tr_R(I)=(t^2, t^3) \ne I$.
\end{ex}

In contrast to the above example, we will obtain that all integrally closed ideals are isomorphic to some trace ideals (Theorem \ref{equiv}). We prepare the following. 

\begin{prop} \label{keyprop}
 Let $M$ be a reflexive module of positive rank. Set $A=\tr_R(M):~\tr_R(M)$. Then, we have $\tr_A(M)=A\cong \tr_R(M)$.
\end{prop}

\begin{proof} 
Set $I=\tr_R(M)$. Then, $I$ is an integrally closed ideal containing a non-zerodivisor of $R$ by Remark \ref{rem2.3}(b) and Proposition \ref{traceint}. We choose $t\in I$ such that $I^2=tI$, and set $A=I:I=t^{-1} I$ by Fact \ref{am1142}. Then, $M$ is an $A$-module by Proposition \ref{reflexive}. We prove that $\tr_A(M)=A$. Indeed, let $f\in \Hom_R(M, R)$. Then, $\Im f\subseteq I$ since $I$ is a trace ideal. We can define the map $g: M \to A$, where $x\mapsto t^{-1}f(x)$. The map $g$ is of course $R$-linear. Furthermore, the following argument proves that $g$ is $A$-linear:
for all $\alpha\in A$ and all $x\in M$, by letting $\alpha=a/s$, where $a,s\in R$ and $s$ is a non-zerodivisor of $R$, we obtain that 
\[
s\alpha g(x)=ag(x)=g(ax)=g(s \alpha x)=s g(\alpha x),
\]
where the fourth equality follows from $\alpha x\in M$ since $M$ is an $A$-module. It follows that $s(\alpha g(x)-g(\alpha x))=0$; hence, $\alpha g(x)=g(\alpha x)$ since $M$ is torsionfree.

Therefore, we obtain that 
\begin{align*} 
\tr_A(M) = \sum_{h\in \Hom_A(M, A)} \Im h \supseteq \sum_{f\in \Hom_R(M, R)} t^{-1}\Im f = t^{-1}\tr_R(M) = A.
\end{align*}
Hence, $\tr_A(M)=A = t^{-1} \tr_R(M) \cong \tr_R(M)$.
\end{proof}

\begin{thm}\label{equiv} 
Let $R$ be an Arf ring. 
For an ideal $I$ containing a non-zerodivisor of $R$, the following are equivalent:
\begin{enumerate}[{\rm (a)}]  
\item $I$ is reflexive as an $R$-module, i.e., $I$ is a divisorial ideal.
\item $I$ is isomorphic to some trace ideal.
\item $I$ is isomorphic to some integrally closed ideal.
\end{enumerate} 
When this is the case, $I\cong \tr_R(I)\cong I^*$ and $\tr_R(I)$ is an integrally closed ideal. 
\end{thm}

\begin{proof} 
(b) $\Rightarrow$ (c) follows from Proposition \ref{traceint}. 
(c) $\Rightarrow$ (a) follows from the following claim.

\begin{claim}\label{claim}
If $I$ is an integrally closed ideal containing a non-zerodivisor of $R$ and $I^2=tI$ for some $t\in I$, then $I=R:(R:I)$.
\end{claim}

\begin{proof}[Proof of Claim \ref{claim}]
Set $A=I:I=t^{-1}I$ (Fact \ref{am1142}). Then, $R:(R:I)=R:(R:tA)=t(R:(R:~A))$. Because $A$ is a subring of $\ol{R}$, we have $(R:A)A{\cdot}A = (R:A)A \subseteq R$, i.e., $(R:A)A \subseteq R:A$. 
Thus $(R:A)A=R:A$. It follows that 
\[
R:(R:A)= R:(R:A)A=(R:A):(R:A) \cong \Hom_R(R:A, R:A).
\] 
Thus, $R:(R:A)$ is a subring of $\ol{R}$; hence, $[R:(R:A)]^2=R:(R:A)$. Hence, $[R:(R:I)]^2=t [R:(R:I)]$. In particular, $(t)$ is a reduction of $R:(R:I)$. Therefore, by \cite[Corollary 1.2.5]{SH}, we obtain that
\[
I \subseteq R:(R:I)\subseteq \ol{(t)}=\ol{I}=I,
\]
which concludes the assertion.
\end{proof}

(a) $\Rightarrow$ (b): Set $A=\tr_R(I):\tr_R(I)$. By Proposition \ref{keyprop}, we have $\tr_A(I)=A$. Hence, by Remark \ref{rem2.3}(a) and Proposition \ref{lindo}, $I$ is a locally free $A$-module of rank one for all maximal ideals of $A$. 
It follows that $I$ is a free $A$-module of rank one because $A$ is semi-local (\cite[Lemma 1.4.4]{BH}). 
Hence, $I\cong A\cong \tr_R(I)$ as $R$-modules. 

When this is the case, $R:I\cong R:\tr_R(I)= \tr_R(I):\tr_R(I)= t^{-1}\tr_R(I) \cong I$, where $t\in \tr_R(I)$ is a reduction of $\tr_R(I)$.
\end{proof}

Now, we are in a position to prove the second main theorem. 

\begin{thm}\label{main2}
Let $(R, \fkm)$ be an Arf local ring. Suppose that $\ol{R}$ is a local ring. If $M$ is a reflexive $R$-module with positive rank, then $M$ decomposes to a direct sum of $\fkm$-primary integrally closed ideals of $R$.
\end{thm}

\begin{proof}
Set $A_1=\tr_R(M): \tr_R(M)$ and $r=\rank_R M$. Note that $A_1$ is local since $\ol{R}$ is local. By Propositions \ref{lindo} and \ref{keyprop}, $M\cong A_1 \oplus N_1$ for some $R$-module $N_1$. $N_1$ is reflexive because $N_1$ is a direct summand of a reflexive module $M$. We also obtain that $N_1$ has rank $r-1$. 
Therefore, by induction on the rank of $M$, we have 
\[
M\cong A_1\oplus A_2 \oplus \cdots \oplus A_r \oplus N_r
\]
where, for $1\le i \le r$, $N_i$ is some reflexive $R$-module of rank $r-i$ and $A_i=\tr_R(N_{i-1}):~\tr_R(N_{i-1})$. Since $N_r$ is reflexive of rank zero, $N_r=0$. Since $A_i\cong \tr_R(N_{i-1})$ by Proposition \ref{keyprop}, we conclude the assertion by Theorem \ref{equiv}.
\end{proof}

As an application of Theorem \ref{main2}, we have the following.

\begin{cor}\label{corcor} 
Let $(R, \fkm)$ be an Arf local domain. Suppose that $\ol{R}$ is a local ring and finitely generated as an $R$-module. Then, there are only finitely many indecomposable reflexive $R$-modules up to isomorphism.
\end{cor}

\begin{proof}
Due to Theorem \ref{main2}, all reflexive modules are decomposed into integrally closed ideals up to isomorphism. On the other hand, there are only finitely many integrally closed ideals up to isomorphism (\cite[Corollary 5.4]{Is}). Therefore, indecomposable reflexive modules are finite up to isomorphism. Because $R$ is a domain, the notion of reflexive modules is equivalent to that of the first syzygies of maximal Cohen-Macaulay modules (e.g., \cite[Lemma 4.1]{HKS}).
\end{proof}

\begin{rem}
Let $\Omega \mathrm{CM}(R)$ denote the category of first syzygies of some maximal Cohen-Macaulay modules. Thus, $\Omega \mathrm{CM}(R)$ consists of all $R$-modules $M$ such that there exists an exact sequence $0\to M \to F \to X \to 0$ with a free $R$-module $F$ and a maximal Cohen-Macaulay $R$-module $X$. 

Suppose that $R$ is a local ring of dimension one and generically Gorenstein, that is, $R_\fkp$ is Gorenstein for all associated prime ideals $\fkp$. Then, it is known that $\Omega \mathrm{CM}(R)$ is the same as the category of reflexive modules (for example, see \cite[Lemma 4.1]{HKS}). 

With this perspective, Corollary \ref{corcor} can be rephrased by saying that $\Omega \mathrm{CM}(R)$ is of finite type. 
\end{rem}


\begin{rem}
The condition on $M$ in Corollary \ref{corcor} cannot be weakened from reflexive to torsionless. Indeed, it is known that for any positive integer $n$, there is an indecomposable maximal Cohen-Macaulay $R$-module of rank $n$ if $R$ is a one-dimensional Cohen-Macaulay local ring with multiplicity $\ge 4$ (\cite[4.2 Theorem]{LW}). On the other hand, we have examples of Arf rings with multiplicity $\ge 4$ (Example \ref{2.13}). 
\end{rem}



\end{document}